\documentclass[11pt]{article}
\usepackage{amsmath}
\usepackage{amsfonts}
\usepackage{amssymb}

\newtheorem{theorem}{Theorem}

\newtheorem{proposition}{Proposition}
\newtheorem{remark}{Remark}
\newenvironment{proof}[1][Proof]{\noindent\textbf{#1.} }{\ \rule{0.5em}{0.5em}}

\begin{document}

\title{\textbf{Bilateral Birth and death process in quantum calculus}}
\date{ }
\author{Lazhar Dhaouadi \thanks{
Institut Pr\'eparatoire aux Etudes d'Ing\'enieur de Bizerte, 7021 Zarzouna,
Tunisia.\quad E-mail: lazhardhaouadi@yahoo.fr}}
\maketitle

\begin{abstract}
In this paper I shall give the complete solution of the equations governing
the bilateral birth and death process on path set $\mathbb{R}_q=\{q^n,\quad
n\in\mathbb{Z}\}$ in which the birth and death rates $\lambda_n=q^{2\nu-2n}$
and $\mu_n=q^{-2n}$ where $0<q<1$ and $\nu>-1$ . The mathematical methods
employed here are based on $q$-Bessel Fourier analysis. \vspace{5mm} \newline
 \vspace{3mm}\newline
\end{abstract}

\section{Introduction}

Birth and death processes were introduced in \cite{FE} by W. Feller in (1939) and have since
been used as models for population growth, queue formation, in epidemiology
and in many other areas of both theoretical and applied interest. From the
standpoint of the theory of stochastic processes they represent an important
special case of Markov processes with countable state spaces and continuous
parameters.\newline

The birth--death process is a special case of continuous-time Markov process
where the state transitions are of only two types: "births", which increase
the state variable by one and "deaths", which decrease the state by one. The
model's name comes from a common application, the use of such models to
represent the current size of a population where the transitions are literal
births and deaths.\newline

The purpose of this paper is to contribute to the knowledge of the
connection between some classe of birth-death processes and the $q$-theory.
We cite for example early result in this direction. The study of the
time-dependent behavior of birth and death processes involves many intricate
and interesting orthogonal polynomials, such as Charlier, Meixner, Laguerre,
Krawtchouk, and other polynomials from the Askey scheme. For example the
authors in \cite[p.350]{Ko} point out that the three-term recurrence
relation of the $q$-Lommel polynomials can be viewed as a three-term
recurrence relation as occurring in birth and death processes with value $%
\lambda_m=w^{-2}q^{-m}$ and $\mu_m=q^{-m}$.\newline

In \cite{B} the authors study the fundamental properties of classical and
quantum Markov processes generated by $q$-Bessel operators and their
extension to the algebra of all bounded operators on the Hilbert space $L^2$%
. They noticed the connection with a bilateral birth and death process
considered in this paper but without an explicite solution to the minimal
semigroup of the classical $q$-Bessel process. But they give an interesting
result about the uniqueness of such semi groupe, which is an important tool in our study.\newline

For birth and death processes with complicated birth and death rates, for
example, when rates are state dependent or nonlinear, it is almost
impossible to find closed form solutions of the transition functions. Due to
the difficulties involved in analytical methods, it is pertinent to develop
other techniques. In this paper, I shall give the complete solution of the
equations governing the bilateral birth and death process on path set $%
\mathbb{R}_q=\{q^n,\quad n\in\mathbb{Z}\}$ in which the birth and death
rates $\lambda_n=q^{2\nu-2n}$ and $\mu_n=q^{-2n}$ where $0<q<1$ and $\nu>-1$%
. The mathematical methods employed here are based on $q$-Bessel Fourier
analysis. In particular the $q$-Bessel operator, $q$-Bessel fourier
transform , $q$-translation operator and $q$-convolution product will appear
in a natural way during our study.\newline

\section{Preliminarily on $q$-Bessel Fourier analysis}

Assume that $0<q<1$ and $\nu >-1$. \ Let $a\in \mathbb{C}$, the $q$-shifted
factorial are defined by
\begin{equation*}
(a;q)_{0}=1,\quad (a;q)_{n}=\prod_{k=0}^{n-1}(1-aq^{k}),\quad (a;q)_{\infty
}=\prod_{k=0}^{\infty }(1-aq^{k}),
\end{equation*}%
and%
\begin{equation*}
\mathbb{R}_{q}=\left\{ q^{n},\quad n\in \mathbb{Z}\right\} .
\end{equation*}%
The $q$-Bessel operator is defined as follows \cite{D1}
\begin{equation*}
\Delta _{q,\nu }f(x)=\frac{1}{x^{2}}\Big[f(q^{-1}x)-(1+q^{2\nu
})f(x)+q^{2\nu }f(qx)\Big].
\end{equation*}%
The eigenfunction of $\Delta _{q,\nu }$ associated with the eigenvalue $%
-\lambda ^{2}$ is the function $x\mapsto j_{\nu }(\lambda x,q^{2})$, where $%
j_{\nu }(.,q^{2})$ is the normalized $q$-Bessel function defined by
\begin{equation*}
j_{\nu }(x,q^{2})=\sum_{n=0}^{\infty }(-1)^{n}\frac{q^{n(n+1)}}{(q^{2\nu
+2},q^{2})_{n}(q^{2},q^{2})_{n}}x^{2n}.
\end{equation*}%
It satisfies the following estimate
\begin{equation}
|j_{\nu }(q^{n},q^{2})|\leq \frac{(-q^{2};q^{2})_{\infty }(-q^{2\nu
+2};q^{2})_{\infty }}{(q^{2\nu +2};q^{2})_{\infty }}\left\{
\begin{array}{c}
1\quad \quad \quad \quad \quad \text{if}\quad n\geq 0 \\
q^{n^{2}-(2\nu +1)n}\quad \text{if}\quad n<0%
\end{array}%
\right. .  \label{e}
\end{equation}%
which show the asymptotic decreasing at infinity on $\mathbb{R}_{q}.$

\bigskip

The $q$-Jackson integral of a function $f$ defined on $\mathbb{R}_{q}$ by
\cite{J}
\begin{equation*}
\int_{0}^{\infty }f(t)d_{q}t=(1-q)\sum_{n\in \mathbb{Z}}q^{n}f(q^{n}).
\end{equation*}%
We denote by $\mathcal{L}_{q,p,\nu }$ the space of functions $f$ defined on $%
\mathbb{R}_{q}$ such that
\begin{equation*}
\Vert f\Vert _{q,p,\nu }=\left[ \int_{0}^{\infty }|f(x)|^{p}x^{2\nu +1}d_{q}x%
\right] ^{1/p}<\infty .
\end{equation*}%
The normalized $q$-Bessel function $j_{\nu }(.,q^{2})$ satisfies the
orthogonality relation \cite{D1}
\begin{equation}
c_{q,\nu }^{2}\int_{0}^{\infty }j_{\nu }(xt,q^{2})j_{\nu }(yt,q^{2})t^{2\nu
+1}d_{q}t=\delta _{q}(x,y),\quad \forall x,y\in \mathbb{R}_{q}^{+}
\label{e2}
\end{equation}%
where
\begin{equation*}
\delta _{q}(x,y)=\left\{
\begin{tabular}{l}
$0$ if $x\neq y$ \\
$\frac{1}{(1-q)x^{2(\nu +1)}}$ if $x=y$%
\end{tabular}%
\ \right. ,
\end{equation*}%
and
\begin{equation*}
c_{q,\nu }=\frac{1}{\left( 1-q\right) }\frac{(q^{2\nu +2},q^{2})_{\infty }}{%
(q^{2},q^{2})_{\infty }}.
\end{equation*}%
Let $f$ be a function defined on $\mathbb{R}_{q}$ then
\begin{equation*}
\int_{0}^{\infty }f(y)\delta _{q}(x,y)y^{2\nu +1}d_{q}y=f(x).
\end{equation*}%
The $q$-Bessel Fourier transform $\mathcal{F}_{q,\nu }$ is defined by \cite%
{D1,K}
\begin{equation*}
\mathcal{F}_{q,\nu }f(x)=c_{q,\nu }\int_{0}^{\infty }f(t)j_{\nu
}(xt,q^{2})t^{2\nu +1}d_{q}t,\quad \forall x\in \mathbb{R}_{q}^{+}.
\end{equation*}

Let $f\in \mathcal{L}_{q,1,\nu }$ then $\mathcal{F}_{q,\nu }f\in \mathcal{C}%
_{q,0}$ and we have
\begin{equation*}
\Vert \mathcal{F}_{q,\nu }f\Vert _{q,\infty }\leq c_{q,\nu }\Vert f\Vert
_{q,1,\nu }.
\end{equation*}%
Let $f$ be a function belongs to $\mathcal{L}_{q,p,\nu }$ where $p\geq 1$
then
\begin{equation}  \label{e4}
\mathcal{F}_{q,\nu }^{2}f=f.
\end{equation}
If $f$ satisfies one of the following conditions :

\begin{description}
\item[i)] $f\in\mathcal{L}_{q,1,\nu }\ $and$\ \mathcal{F}_{q,\nu }f\in
\mathcal{L}_{q,1,\nu }.$

\item[ii)] $f\in\mathcal{L}_{q,1,\nu }\cap \mathcal{L}_{q,2,\nu }$ where $%
p>2.$

\item[iii)] $f\in\mathcal{L}_{q,2,\nu }.$
\end{description}

Then we have $\Vert \mathcal{F}_{q,\nu }f\Vert _{q,2,\nu }=\Vert f\Vert
_{q,2,\nu }$. Note that if we denote by
\begin{equation*}
\langle f,g\rangle _{q,\nu }=\int_{0}^{\infty }f(x)g(x)x^{2\nu +1}d_{q}x
\end{equation*}%
the inner product in the space $\mathcal{L}_{q,2,\nu }$ then we have
\begin{equation*}
\Big\langle\mathcal{F}_{q,\nu }(f),\mathcal{F}_{q,\nu }(g)\Big\rangle_{q,\nu
}=\langle f,g\rangle _{q,\nu},\quad \forall f,g\in \mathcal{L}_{q,2,\nu}.
\end{equation*}

\bigskip

The $q$-translation operator is defined by
\begin{equation*}
T_{q,x}^{\nu }f(y)=c_{q,\nu }\int_{0}^{\infty }\mathcal{F}_{q,\nu
}f(t)j_{\nu }(yt,q^{2})j_{\nu }(xt,q^{2})t^{2\nu +1}d_{q}t.
\end{equation*}
Let us now introduce
\begin{equation*}
Q_\nu=\Big\{q\in ]0,1[,\quad T^\nu_{q,x}\quad \text{is positive for all}%
\quad x\in\mathbb{R}_q\Big\},
\end{equation*}
the set of the positivity of $T^\nu_{q,x}$. We recall that $T^\nu_{q,x}$ is
called positive if $T^\nu_{q,x}f\geq 0$ for $f\geq 0$.\newline

In \cite{FD} it was proved that if $-1<\nu<\nu^{\prime }$ then $%
Q_{\nu}\subset Q_{\nu^{\prime }}$. As a consequence :

\begin{description}
\item[-] If $0\leq \nu$ then $Q_{\nu}=]0,1[.$

\item[-] If $-\frac{1}{2}\leq \nu<0$ then $]0,q_{0}]\subset Q_{-\frac{1}{2}%
}\subset $ $Q_{\nu}\subsetneq]0,1[$,\quad $q_0\simeq 0.43$.

\item[-] If $-1<\nu\leq -\frac{1}{2}$ then $Q_{\nu}\subset Q_{-\frac{1}{2}}.$
\end{description}

Let $f\in \mathcal{L}_{q,p,\nu }$ then $T_{q,x}^{\nu }f$ exist and we have
\begin{equation*}
\int_{0}^{\infty }T_{q,x}^{\nu }f(y)y^{2\nu +1}d_{q}y=\int_{0}^{\infty
}f(y)y^{2\nu +1}d_{q}y.
\end{equation*}%
If we assume that $T_{q,x}^{\nu }$ is positif then
\begin{equation*}
\Vert T_{q,x}^{\nu }f\Vert _{q,p,\nu }\leq \Vert f\Vert _{q,p,\nu }.
\end{equation*}
The $q$-convolution product is given as follows \cite{D1}
\begin{equation*}
f\ast _{q}g=\mathcal{F}_{q,\nu }\left[ \mathcal{F}_{q,\nu }f\times \mathcal{F%
}_{q,\nu }g\right] .
\end{equation*}
Let $1\leq p\leq 2$ and $1\leq r,s$ such that
\begin{equation*}
\frac{1}{p}+\frac{1}{r}-1=\frac{1}{s}.
\end{equation*}%
If $f\in \mathcal{L}_{q,p,\nu }$ and $g\in \mathcal{L}_{q,r,\nu }$ then $%
f\ast _{q}g$ exist and we have
\begin{equation*}
f\ast _{q}g(x)=c_{q,\nu }\int_{0}^{\infty }T_{q,x}^{\nu }f(y)g(y)y^{2\nu
+1}d_{q}y.
\end{equation*}%
In addition, if $1\leq r\leq 2$ then
\begin{equation*}
\mathcal{F}_{q,\nu }(f\ast _{q}g)=\mathcal{F}_{q,\nu }(f)\times \mathcal{F}%
_{q,\nu }(g).
\end{equation*}%
If $s\geq 2$ then $f\ast _{q}g\in \mathcal{L}_{q,s,\nu }$ and
\begin{equation}
\Vert f\ast _{q}g\Vert _{q,s,\nu }\leq c_{q,\nu }\Vert f\Vert _{q,p,\nu
}\times \Vert g\Vert _{q,r,\nu }  \label{e1}
\end{equation}%
If we assume that $T_{q,x}^{\nu }$ is positif then (\ref{e1}) hold true for
all $s\geq 1$.\newline

In the sequel, we will always assume $q\in Q_\nu$.\newline

\section{Bilateral birth and death processes on $\mathbb{R}_q$}

We consider a bilateral birth and death processes $X_t$ with parameter set ${%
\mathcal{T }}=[0,\infty)$ on the path set
\begin{equation*}
\mathbb{R}_q=\{q^n,\quad n\in\mathbb{Z}\}
\end{equation*}
with stationary transition probabilities
\begin{equation*}
p_{i,j}(h)=\mathrm{Pr}\Big[X_{t+h}=q^j\Big|X_t=q^i\Big]
\end{equation*}
which is not depending on $t$. In addition we assume that $p_{i,j}(h)$
satisfy

\begin{itemize}
\item $p_{i,i+1}(h)=q^{2\nu-2i}h+o(h)$ as $h\downarrow 0$, \quad $\forall
i\in\mathbb{Z}$.

\item $p_{i,i-1}(h)=q^{-2i}h+o(h)$ as $h\downarrow 0$, \quad $\forall i\in%
\mathbb{Z}$.

\item $p_{i,i}=1-(q^{2\nu-2i}+q^{-2i})h+o(h)$ as $h\downarrow 0$, \quad $%
\forall i\in\mathbb{Z}$.

\item $p_{i,j}(h)=0$ if $|i-j|>1$.
\end{itemize}

The state space of this process is interpreted as jump rate from the point $%
q^{i}$ to $q^{i+1}$ or $q^{i-1}$. The process will be parameterized by a
continuous time $t$, but its trajectories will not be continuous.

\bigskip

Fixe an arbitrary state $q^{r}$ and let
\begin{equation*}
p_{nr}(t)=\mathrm{Pr}\left[ \left. X_{t}=q^{n}\right\vert X_{0}=q^{r}\right]
.
\end{equation*}%
We use the following notation if there is no confusion
\begin{equation*}
p_{n}(t)=p_{nr}(t).
\end{equation*}%
Then we obtain%
\begin{eqnarray*}
p_{n}(t+h) &=&\mathrm{Pr}\left[ \left. X_{t+h}=q^{n}\right\vert X_{0}=q^{r}%
\right] \\
&=&\mathrm{Pr}\left[ \left. X_{t}=q^{n-1}\right\vert X_{0}=q^{j}\right]
p_{n-1,n}(h)+\mathrm{Pr}\left[ \left. X_{t}=q^{n}\right\vert X_{0}=q^{r}%
\right] p_{n,n}(h)+ \\
&&\mathrm{Pr}\left[ \left. X_{t}=q^{n+1}\right\vert X_{0}=q^{r}\right]
p_{n+1,n}(h) \\
&=&q^{2\nu -2(n-1)}hp_{n-1}(t)+\left[ 1-(q^{2\nu -2n}+q^{-2n})h\right]
p_{n}(t)+q^{-2(n+1)}hp_{n+1}(t)+o(h).
\end{eqnarray*}%
Taking the limit $h\rightarrow 0$ we must have the following differential
equation
\begin{equation*}
\frac{d}{dt}p_{n}(t)=q^{2\nu -2(n-1)}p_{n-1}(t)-(q^{2\nu
-2n}+q^{-2n})p_{n}(t)+q^{-2(n+1)}p_{n+1}(t).
\end{equation*}%
Let $x_{n}=q^{n}$, $x_{r}=q^{r}$ and
\begin{equation*}
P_{x_{r}}(x_{n},t)=\frac{1}{(1-q)}q^{-2(\nu +1)n}p_{n}(t)=\frac{1}{(1-q)}%
x_{n}^{-2(\nu +1)}p_{n}(t),
\end{equation*}%
We obtain the following equation :
\begin{eqnarray*}
\frac{d}{dt}P_{x_{r}}(x_{n},t) &=&\frac{P_{x_{r}}(x_{n-1},t)-(1+q^{2\nu
})P_{x_{r}}(x_{n},t)+q^{2\nu }P_{x_{r}}(x_{n+1},t)}{x_{n}^{2}} \\
&=&\frac{P_{x_{r}}(q^{-1}x_{n},t)-(1+q^{2\nu })P_{x_{r}}(x_{n},t)+q^{2\nu
}P_{x_{r}}(qx_{n},t)}{x_{n}^{2}}.
\end{eqnarray*}%
Replacing $x_{n}$ by an arbitrary $x\in \mathbb{R}_{q}$ we deduce the $q$%
-Fokker-Planck equation
\begin{equation*}
\frac{\partial }{\partial t}P_{x_{r}}(x,t)=\Delta _{q,\nu }P_{x_{r}}(x,t).
\end{equation*}%
The solution is explicitly written as follows
\begin{align*}
P_{x_{r}}(x,t)& =c_{q,\nu }^{2}\int_{0}^{\infty }e^{-ty^{2}}j_{\nu
}(xy,q^{2})j_{\nu }(x_{r}y,q^{2})y^{2\nu +1}d_{q}y \\
& =c_{q,\nu }\mathcal{F}_{q,\nu }\left[ z\rightarrow e^{-tz^{2}}j_{\nu
}(x_{r}z,q^{2})\right] (x) \\
& =c_{q,\nu }T_{q,x_{r}}^{\nu }\mathcal{F}_{q,\nu }\left[ z\rightarrow
e^{-tz^{2}}\right] (x). \\
& =c_{q,\nu }T_{q,x_{r}}^{\nu }\mathcal{\rho }_{t}(x),
\end{align*}
where%
\begin{equation*}
\mathcal{\rho }_{t}(x)=\mathcal{F}_{q,\nu }\left[ z\rightarrow e^{-tz^{2}}%
\right] (x).
\end{equation*}%
In fact
\begin{eqnarray*}
\frac{\partial }{\partial t}P_{x_{r}}(x,t) &=&c_{q,\nu }^{2}\int_{0}^{\infty
}-y^{2}e^{-ty^{2}}j_{\nu }(xy,q^{2})j_{\nu }(x_{r}y,q^{2})y^{2\nu +1}d_{q}y
\\
&=&c_{q,\nu }^{2}\int_{0}^{\infty }e^{-ty^{2}}\left[ \Delta _{q,\nu }j_{\nu
}(xy,q^{2})\right] j_{\nu }(x_{r}y,q^{2})y^{2\nu +1}d_{q}y \\
&=&\Delta _{q,\nu }\left[ c_{q,\nu }^{2}\int_{0}^{\infty }e^{-ty^{2}}j_{\nu
}(xy,q^{2})j_{\nu }(x_{r}y,q^{2})y^{2\nu +1}d_{q}y\right] \\
&=&\Delta _{q,\nu }P_{x_{r}}(x,t).
\end{eqnarray*}
The derivative under integral sign follows from (\ref{e}). Hence
\begin{equation}  \label{e3}
p_{nr}(t)=(1-q)q^{2(\nu +1)n}P_{x_{r}}(x_{n},t)
\end{equation}

\begin{proposition}
The stationary transition probabilities $p_{nr}(t)$ satisfies :

\begin{description}
\item[a.] $p_{nr}(t)\geq 0$.

\item[b.] $\sum_{n\in \mathbb{Z}}p_{nr}(t)=1$.

\item[c.] $p_{nr}(0)=\delta (n,r).$

\item[d.] $p_{nr}(t+s)=\sum_{k\in \mathbb{Z}}p_{nk}(t)p_{kr}(s).$
\end{description}
\end{proposition}

\begin{proof}
In fact
\begin{equation*}
p_{nr}(t)=(1-q)q^{2(\nu +1)n}P_{x_{r}}(x_{n},t)=c_{q,\nu }(1-q)q^{2(\nu
+1)n}T_{q,x_{r}}^{\nu }\mathcal{F}_{q,\nu }\left[ z\rightarrow e^{-tz^{2}}%
\right] (x_{n}).
\end{equation*}%
In the proof of \cite[Theroem 1]{D2} we have proved that $\mathcal{F}_{q,\nu
}\left[ e^{-tz^{2}}\right] $ is a positive function. The Positivity of the
generalized translation operator $T_{q,x_{0}}$ implies that $p_{nr}(t)\geq 0$%
.

\bigskip

It also satisfies the normalization of the total probability :
\begin{align*}
\sum_{n\in \mathbb{Z}}p_{nr}(t)& =(1-q)\sum_{n\in \mathbb{Z}}x_{n}^{2(\nu
+1)}P_{x_{r}}(x_{n},t) \\
& =\int_{0}^{\infty }P_{x_{r}}(x,t)x^{2\nu +1}d_{q}x \\
& =c_{q,\nu }\int_{0}^{\infty }T_{q,x_{r}}^{\nu }\mathcal{F}_{q,\nu }\left[
z\rightarrow e^{-tz^{2}}\right] (x)x^{2\nu +1}d_{q}x \\
& =c_{q,\nu }\int_{0}^{\infty }\mathcal{F}_{q,\nu }\left[ z\rightarrow
e^{-tz^{2}}\right] (x)x^{2\nu +1}d_{q}x \\
& =\mathcal{F}_{q,\nu }^{2}[z\rightarrow e^{-tz^{2}}](0)=e^{-ty^{2}}\Bigg|%
_{y=0}=1.
\end{align*}%
To show that it satisfies the initial conditions we use formula (\ref{e2}) :
\begin{eqnarray*}
p_{nr}(0) &=&(1-q)x_{n}^{2(\nu +1)}P_{x_{r}}(x_{n},0) \\
&=&(1-q)x_{n}^{2(\nu +1)}c_{q,\nu }^{2}\int_{0}^{\infty }j_{\nu
}(x_{n}y,q^{2})j_{\nu }(x_{r}y,q^{2})y^{2\nu +1}d_{q}y \\
&=&(1-q)x_{n}^{2(\nu +1)}\delta _{q}(x_{r},x_{n})=\delta (n,r).
\end{eqnarray*}%
The processes satisfy the markovian property :
\begin{eqnarray*}
&&\sum_{k\in \mathbb{Z}}p_{nk}(t)p_{kr}(s) \\
&=&(1-q)^{2}c_{q,\nu }^{2}x_{r}^{2(\nu +1)}\sum_{k\in \mathbb{Z}%
}x_{k}^{2(\nu +1)}\mathcal{F}_{q,\nu }\left[ z\rightarrow e^{-tz^{2}}j_{\nu
}(x_{n}z,q^{2})\right] (x_{k})\mathcal{F}_{q,\nu }\left[ z\rightarrow
e^{-sz^{2}}j_{\nu }(x_{r}z,q^{2})\right] (x_{k}) \\
&=&(1-q)c_{q,\nu }^{2}x_{r}^{2(\nu +1)}\int_{0}^{\infty }\mathcal{F}_{q,\nu }%
\left[ z\rightarrow e^{-tz^{2}}j_{\nu }(x_{n}z,q^{2})\right] (x)\mathcal{F}%
_{q,\nu }\left[ z\rightarrow e^{-sz^{2}}j_{\nu }(x_{r}z,q^{2})\right]
(x)x^{2\nu +1}d_{q}x \\
&=&(1-q)c_{q,\nu }^{2}x_{r}^{2(\nu +1)}\left\langle \mathcal{F}_{q,\nu }%
\left[ z\rightarrow e^{-tz^{2}}j_{\nu }(x_{n}z,q^{2})\right] ,\mathcal{F}%
_{q,\nu }\left[ z\rightarrow e^{-sz^{2}}j_{\nu }(x_{r}z,q^{2})\right]
\right\rangle _{q,\nu } \\
&=&(1-q)x_{r}^{2(\nu +1)}c_{q,\nu }^{2}\left\langle e^{-tz^{2}}j_{\nu
}(x_{n}z,q^{2}),e^{-sz^{2}}j_{\nu }(x_{r}z,q^{2})\right\rangle _{q,\nu } \\
&=&(1-q)x_{r}^{2(\nu +1)}c_{q,\nu }^{2}\int_{0}^{\infty
}e^{-(t+s)z^{2}}j_{\nu }(x_{n}z,q^{2})j_{\nu }(x_{r}z,q^{2})z^{2\nu +1}d_{q}z
\\
&=&p_{nr}(t+s).
\end{eqnarray*}
\end{proof}

\begin{proposition}
The solution given by (\ref{e3}) is unique if and only if $\nu\geq 0$.
\end{proposition}

\begin{proof}
In fact in \cite[Theorem 3.5]{P} it was given a necessary and sufficient
condition that there is one and only one solution of a stationary transition
probabilities $p_{nr}(t)$ satisfies the conditions of Proposition 1. In our
situation the result is an immediate consequence of the uniqueness result
given by \cite[Theorem 2]{B}.
\end{proof}

\section{The transition semigroup}

To introduce the transition semigroup we follows that given in \cite{KM}.
But in our case this is a bilateral birth and death processes on the path
set $\mathbb{R}_{q}$ which must replace $\mathbb{N}$.\newline

We present here some connexions between the notations used in \cite{KM} and
the standard notations of $q$-Bessel Fourier analysis. Let
\begin{equation*}
\pi _{0}=1,\text{ \ \ \ \ }\pi _{n}=\frac{\lambda _{0}\ldots \lambda _{n-1}}{%
\mu _{1}\ldots \mu _{n}},\text{ \ \ }\forall n\in \mathbb{Z}^{\ast }
\end{equation*}%
In our case
\begin{equation*}
\lambda _{n}=q^{2\nu -2n}\text{ \ and }\mu _{n}=q^{-2n}\Rightarrow \pi
_{n}=q^{2(\nu +1)n}
\end{equation*}%
The $L^{2}(\pi )$ norm was introduced in \cite{KM} as follows
\begin{equation*}
\left\Vert f\right\Vert ^{2}=\sum_{n\in \mathbb{Z}}\left\vert
f(q^{n})\right\vert ^{2}\pi _{n}=\sum_{n\in \mathbb{Z}}\left\vert
f(q^{n})\right\vert ^{2}q^{2(\nu +1)n}=\left\Vert f\right\Vert _{q,2,\nu
}^{2}\Rightarrow L^{2}(\pi )=\mathcal{L}_{q,2,\nu }.
\end{equation*}%
Also for a given functions $f$ and $g$ in $L^{2}(\pi )$ we have
\begin{equation*}
\left\langle f,g\right\rangle =\sum_{n\in \mathbb{Z}}f(q^{n})\overline{%
g(q^{n})}\pi _{n}=\int_{0}^{\infty }f(x)\overline{g(x)}x^{2\nu
+1}d_{q}x=\left\langle f,g\right\rangle _{q,\nu }.
\end{equation*}%
Now looking at the operator $A$ defined by
\begin{equation*}
Af(q^{n})=\sum_{k\in \mathbb{Z}}a_{n,k}f(q^{k})
\end{equation*}%
where%
\begin{equation*}
a_{n,k}=\left\{
\begin{array}{c}
\mu _{n}\text{ \ if }k=n-1 \\
-(\lambda _{n}+\mu _{n})\text{ if }k=n \\
\lambda _{n}\text{\ if }k=n+1 \\
0\text{ \ otherwise}%
\end{array}%
\right. .
\end{equation*}%
A simple calculation
\begin{eqnarray*}
Af(q^{n}) &=&a_{n,n-1}f(q^{n-1})+a_{n,n}f(q^{n})+a_{n,n+1}f(q^{n+1}) \\
&=&\mu _{n}f(q^{n-1})-(\lambda _{n}+\mu _{n})f(q^{n})+\lambda _{n}f(q^{n+1})
\\
&=&q^{-2n}f(q^{n-1})-(q^{2\nu -2n}+q^{-2n})f(q^{n})+q^{2\nu -2n}f(q^{n+1}) \\
&=&\frac{f(q^{n-1})-(1+q^{2\nu })f(q^{n})+q^{2\nu }f(q^{n+1})}{q^{2n}},
\end{eqnarray*}%
lead to the fact that when $x=q^{n}$ we obtain
\begin{equation*}
Af(x)=\Delta _{q,\nu }f(x)\Rightarrow A=\Delta _{q,\nu }.
\end{equation*}%
Then
\begin{equation*}
\left\langle Af,g\right\rangle =\left\langle f,Ag\right\rangle \Rightarrow
\left\langle \Delta _{q,\nu }f,g\right\rangle _{q,\nu }=\left\langle
f,\Delta _{q,\nu }g\right\rangle _{q,\nu }.
\end{equation*}%
The operator $T_{t}$ introduced in \cite[p. 518]{KM} seem to be the
appropriate choies for the transition semigroup.

In our study we use $P_t$ instead of $T_t$. Let $P_{t}f(x)$ given for $t\geq
0$ and $x=q^{r}$ by
\begin{eqnarray*}
P_{t}f(x) &=&\mathbb{E}\left[ f(X_{t})\right] =\sum_{n\in \mathbb{Z}}\mathrm{%
Pr}\left[ \left. X_{t}=q^{n}\right\vert X_{0}=q^{r}\right] f(q^{n}) \\
&=&\sum_{n\in \mathbb{Z}}p_{nr}(t)f(q^{n}) \\
&=&(1-q)\sum_{n\in \mathbb{Z}}q^{2(\nu +1)n}P_{x}(q^{n},t)f(q^{n}) \\
&=&(1-q)c_{q,\nu }\sum_{n\in \mathbb{Z}}q^{2(\nu +1)n}T_{q,x}^{\nu }\mathcal{%
\rho }_{t}(q^{n})f(q^{n}) \\
&=&c_{q,\nu }\int_{0}^{\infty }T_{q,x}^{\nu }\mathcal{\rho }%
_{t}(y)f(y)y^{2\nu +1}d_{q}y \\
&=&\mathcal{\rho }_{t}\ast _{q}f(x),
\end{eqnarray*}%
with initial condition
\begin{equation*}
P_{0}f(x)=\sum_{n\in \mathbb{Z}}p_{nr}(0)f(q^{n})=\sum_{n\in \mathbb{Z}%
}\delta (r,n)f(q^{n})=f(q^{r})=f(x).
\end{equation*}%
From Proposition 1 we have the normalisation of total probability
\begin{equation*}
P_{t}1=\sum_{n\in \mathbb{Z}}p_{nr}(t)=1,
\end{equation*}
and the positivity
\begin{equation*}
f\geq0\Rightarrow P_tf\geq 0.
\end{equation*}
From \cite{KM} we see that $P_{t}$ define a bounded linear self-adjoint
operator of $L^{2}(\pi)$ into itself. The mapping
\begin{equation*}
t\mapsto P_{t}f
\end{equation*}%
is continuous on $0\leq t<\infty $ relative to the strong operator topology.
Also by the use of Proposition 1 we have
\begin{eqnarray*}
P_{t}P_{s}f(x) &=&\sum_{n\in \mathbb{Z}}p_{nr}(t)P_{s}f(q^{n}) \\
&=&\sum_{n\in \mathbb{Z}}p_{nr}(t)\left[ \sum_{k\in \mathbb{Z}%
}p_{kn}(s)f(q^{k})\right] =\sum_{n\in \mathbb{Z}}\left( \sum_{k\in \mathbb{Z}%
}p_{nr}(t)p_{kn}(s)\right) f(q^{k}) \\
&=&\sum_{k\in \mathbb{Z}}\left( \sum_{n\in \mathbb{Z}}p_{kn}(s)p_{nr}(t)%
\right) f(q^{k})=\sum_{k\in \mathbb{Z}}p_{kr}(t+s)f(q^{k}) \\
&=&P_{t+s}f(x),
\end{eqnarray*}
whenever $f$ and $g$ are with compact support. But functions with compact
support are \ dense in $L^{2}(\pi)$ and hence the semi group property
established in $L^{2}(\pi)$
\begin{equation*}
P_{t}P_{s}f=P_{t+s}f\text{ \ \ \ \ ,}\forall f\in L^{2}(\pi).
\end{equation*}%
A direct consequence is the fact that $P_{t}$ is positive definite
\begin{equation*}
\left\langle P_{t}f,f\right\rangle _{q,\nu }\geq 0\text{ \ \ \ ,}\forall
f\in L^{2}(\pi).
\end{equation*}

\begin{remark}
The operator $P_{t}$ was introduce in \cite{FD} and many of it's properties
was established.
\end{remark}

\begin{theorem}
A solution to the following $q$-heat equation:
\begin{equation*}
\frac{\partial }{\partial t}u(t,x)=\Delta _{q,\nu }u(t,x)
\end{equation*}%
with initial condition
\begin{equation*}
u(0,x)=f(x),\quad f\in\mathcal{L}_{q,p,\nu},
\end{equation*}%
is given by $u(t,x)=P_{t}f(x),\quad \forall x\in \mathbb{R}_{q}$. If $f\in%
\mathcal{L}_{q,2,\nu}$ then there exists a unique solution if and only if $%
\nu \geq 0.$
\end{theorem}

\begin{proof}
In fact let%
\begin{equation*}
u(t,x)=P_{t}f(x)=\mathcal{\rho }_{t}\ast _{q}f(x)
\end{equation*}%
Using the properties of the $q$-convolution product we obtain
\begin{eqnarray*}
u(t,x) &=&\mathcal{F}_{q,\nu }\left[ \mathcal{F}_{q,\nu }\mathcal{\rho }%
_{t}\times \mathcal{F}_{q,\nu }f\right] (x) \\
&=&c_{q,\nu }\int_{0}^{\infty }\mathcal{F}_{q,\nu }\mathcal{\rho }_{t}(y)%
\mathcal{F}_{q,\nu }f(y)j_{\nu }(xy,q^{2})y^{2\nu +1}d_{q}y \\
&=&c_{q,\nu }\int_{0}^{\infty }e^{-ty^{2}}\mathcal{F}_{q,\nu }f(y)j_{\nu
}(xy,q^{2})y^{2\nu +1}d_{q}y.
\end{eqnarray*}%
The inversion formula (\ref{e4}) lead to the initial condition
\begin{equation*}
u(0,x)=c_{q,\nu }\int_{0}^{\infty }\mathcal{F}_{q,\nu }f(y)j_{\nu
}(xy,q^{2})y^{2\nu +1}d_{q}y=\mathcal{F}_{q,\nu }^{2}f(x)=f(x).
\end{equation*}%
On the other hand%
\begin{eqnarray*}
\frac{\partial }{\partial t}u(t,x) &=&c_{q,\nu }\int_{0}^{\infty }\left(
-y^{2}\right) e^{-ty^{2}}\mathcal{F}_{q,\nu }f(y)j_{\nu }(xy,q^{2})y^{2\nu
+1}d_{q}y. \\
&=&c_{q,\nu }\int_{0}^{\infty }e^{-ty^{2}}\mathcal{F}_{q,\nu }f(y)\left[
\Delta _{q,\nu }j_{\nu }(xy,q^{2})\right] y^{2\nu +1}d_{q}y \\
&=&\Delta _{q,\nu }u(t,x).
\end{eqnarray*}%
The uniqueness is a consequence of Proposition 2.
\end{proof}

\end{document}